\documentclass[11pt]{article}
\usepackage{graphicx} 
\usepackage{bbm}
\usepackage{geometry}
\usepackage{color}
\usepackage{amsfonts}
\usepackage{algorithm}
\usepackage{algorithmic}
\usepackage{latexsym, amssymb, amsmath, amscd, amsthm, amsxtra}
\usepackage{mathtools}
\usepackage{enumerate}
\usepackage[all]{xy}
\usepackage{mathrsfs}
\usepackage{fancyhdr}
\usepackage{listings}
\usepackage{hyperref}
\usepackage{enumitem}

\def\11{\mathbf{1}}

\def\ER{Erd\H{o}s-R\'enyi\ }

\newcommand{\Pb}{\mathbb P}
\newcommand{\Qb}{\mathbb Q}

\newtheorem{thm}{Theorem}[section]

\newtheorem{proposition}[thm]{Proposition}

\newtheorem{lemma}[thm]{Lemma}

\newtheorem{remark}[thm]{Remark}

\numberwithin{equation}{section}



\hypersetup{colorlinks=true,linkcolor=blue}
\title{Strong Detection Threshold for Correlated \ER Graphs with Constant Average Degree}
\author{Chenxu Feng\\Peking University}
\date{\today}
\begin{document}
\maketitle

\begin{abstract}
    Consider a pair of correlated \ER graphs $\mathcal G(n,\tfrac{\lambda}{n};s)$ that are subsampled from a common parent \ER graph with average degree $\lambda$ and subsampling probability $s$. We establish a sharp information-theoretic threshold for the detection problem between this model and two independent \ER graphs $\mathcal G(n,\tfrac{\lambda}{n})$, showing that strong detection is information-theoretically possible if and only if $s>\min\{ \tfrac{1}{\sqrt{\lambda}}, \sqrt{\alpha} \}$ where $\alpha\approx 0.338$ is the Otter's constant. Our result resolves a constant gap between \cite{DD23a} and \cite{WXY23}.
\end{abstract}

\section{Introduction}

In this paper, we study the information-theoretic threshold for detecting the correlation between a pair of sparse \ER graphs. To be mathematically precise, we first need to choose a model for a pair of correlated \ER graphs, and a natural choice is that the two graphs are independently sub-sampled from a common \ER graph, as describe more precisely next. 
For any integer $n$, denote by $\operatorname{U}_n$ the set of unordered pairs $(i,j)$ with $1\le i\neq j\le n$. Given two model parameters $\lambda,s\in (0,1)$, for $(i,j) \in \operatorname{U}_n$ let $I_{i,j}$ be independent Bernoulli variables with parameter $\tfrac{\lambda}{n}$, and let $J_{i,j}$ and $K_{i,j}$ be independent Bernoulli variables with parameter $s$. In addition, let $\pi_*$ be an independent uniform permutation on $[n]=\{1,\dots,n\}$. Then, we define a triple of correlated random graphs $(G,A,B)$ such that for $(i,j) \in \operatorname{U}_n$ (note that we identify a graph with its adjacency matrix)
\begin{align*}
    G_{i,j}=I_{i,j}, A_{i,j}=I_{i,j}J_{i,j}, B_{\pi_*(i),\pi_*(j)}=I_{i,j}K_{i,j}\,.
\end{align*}
We denote $\mathcal G(n,\tfrac{\lambda}{n};s)$ as the joint law of $(A,B)$. It is obvious that marginally $A$ is an \ER graph with edge density $\tfrac{\lambda s}{n}$ and so is $B$, whose law we denote as $\mathcal G(n,\tfrac{\lambda s}{n})$. In addition, the edge correlation $\rho$ is given by $\rho = s(1-\tfrac{\lambda}{n})/(1-\tfrac{\lambda s}{n})$. 

Given two graphs $(A,B)$, one natural question is the following hypothesis testing problem, where we need to decide whether $(A,B)$ are sampled from the law of two correlated \ER graphs $\mathcal G(n,\tfrac{\lambda}{n};s)$ (denoted as $\Pb$) or are sampled from two independent \ER graphs $\mathcal G(n,\tfrac{\lambda s}{n})$ (denoted as $\Qb$) . In the super-constant degree region where $\lambda=\omega(1)$, it was established in \cite{DD23a} that there is a sharp phase transition for this detection problem at $\lambda s^2=C_*(\lambda)$ for a specific constant $C_*(\lambda)$. They also note that when $\lambda$ is a constant such sharp phase transition phenomenon no longer exists, as one can always get a statistics better than random guessing as long as $s$ is a constant. Thus, in this paper we will focus on the strong detection problem when $\lambda=O(1)$ where the question is whether we can solve this detection problem with vanishing type-I and type-II errors. Our result can be summarized as follows:

\begin{thm}{\label{MAIN-THM}}
    Suppose $\lambda$ is a constant. Then for any constant $\epsilon>0$, 
    \begin{enumerate}
        \item[(1)] if $s<\min\{ \tfrac{1}{\sqrt{\lambda}}, \sqrt{\alpha} \}-\epsilon$ where $\alpha\approx 0.338$ is the Otter's threshold, we have $\operatorname{TV}(\Pb,\Qb)=1-\Omega(1)$ and thus strong-detection is information-theoretically impossible. 
        \item[(2)] if $s>\min\{ \tfrac{1}{\sqrt{\lambda}}, \sqrt{\alpha} \}+\epsilon$, we have $\operatorname{TV}(\Pb,\Qb)=1-o(1)$ and thus strong-detection is information-theoretically possible. 
    \end{enumerate}
\end{thm}
\begin{remark}{\label{rmk-previous-lower-bound}}
    We note that Result~(2) can already been derived by combining the results of \cite{DD23a} and \cite{MWXY21+}. Specifically, when $\lambda s^2>1$, in \cite{DD23a} the authors propose a statistic that achieves strong detection based on densest subgraphs. In addition, when $s^2>\alpha$ in \cite{MWXY21+} the authors propose a statistic that achieves strong detection based on counting trees. However, the best known lower bound, established in \cite{WXY23}, only shows that strong detection is information-theoretically impossible when $s<\min\{ \tfrac{1}{\sqrt{\lambda}}, 0.1 \}$. Thus, our result sharpens a constant factor to the previous best lower bound.
\end{remark}

\begin{remark}
    Our result can be easily adapted to the regime where $n^{-1+o(1)}\leq \lambda =o(1)$, where the lower bound is just given by $\sqrt{\alpha}$ (since $\tfrac{1}{\sqrt{\lambda}} \to \infty$ in this regime). This also matches the upper bound in \cite{MWXY21+} where they achieve strong detection above the Otter's threshold by counting trees.
\end{remark}

\subsection{Backgrounds and related works}

Graph matching refers to finding the vertex correspondence between two graphs such that the total number of common edges is maximized. It plays essential roles in various applied fields such as computational biology \cite{SXB08, VCL+15}, social networking \cite{NS08, NS09}, computer vision \cite{BBM05, CSS06} and natural language processing \cite{HNM05}. This problem is an instance of {\em Quadratic assignment problem} which is known to be NP-hard in the worst case, making itself an important and challenging combinatorial optimization problem. The correlated \ER graph model, introduced in \cite{PG11}, provides a canonical probabilistic model for studying graph matching problem in the average case. In recent years, many results have been obtained concerning this model with emphasis placed on the two important and entangling issues: the information threshold (i.e., the statistical threshold) and the computational phase transition. On one hand, the collaborative endeavors of the community, as evidenced in \cite{CK16, CK17, HM23, GML21, WXY22, WXY23, DD23a, DD23b, Du25+}, have led to significant developments on the understanding of information thresholds in correlation detection and vertex matching. On the other hand, continual advancements in algorithms for both problems are evident in works such as \cite{BCL+19, DMWX21, FMWX23a, FMWX23b, BSH19, CKMP19, DCKG19, MX20, GM20, GML20+, MRT21, MRT23, MWXY21+, GMS22+, MWXY23, DL22+, DL23+}. The current state of the art in algorithms can be summarized as follows: in the sparse regime, efficient matching algorithms are available when the correlation exceeds the square root of Otter’s constant (which is approximately 0.338) \cite{GML20+, GMS22+, MWXY21+, MWXY23}; in the dense regime, efficient matching algorithms exist as long as the correlation exceeds an arbitrarily small constant \cite{DL22+, DL23+}. Roughly speaking, the separation between the sparse and dense regimes mentioned above depends on whether the average degree grows polynomially or sub-polynomially. In addition, although it seems rather elusive to prove hardness on the graph matching problem for a typical instance under the assumption of P$\neq$NP, in \cite{DDL23+, Li25+} the authors presented evidences that the state-of-the-art algorithms indeed capture the correct computation thresholds based on the analysis of a specific class known as low-degree polynomial algorithms. 

We now turn our attention to the correlated \ER model in the constant-degree regime, where the average degree satisfies $\lambda=O(1)$. Compared to the extensively studied super-constant degree regime, our understanding of this regime remains relatively limited. Although various developments have been made by the community on both the information theoretic and the computational aspects of graph matching problem in this regime (see, e.g., \cite{HM23, GML21, GML20+, WXY23, GMS22+}), as discussed in Remark~\ref{rmk-previous-lower-bound} there is still a constant gap in determining the sharp threshold for strong detection in this model. Our result resolve this gap, and we will discuss some key innovations and interesting features revealed in this work below. 
\begin{itemize}
    \item {\bf Information-computation gap.} As shown in Theorem~\ref{MAIN-THM}, when $\lambda$ is a constant, strong detection between $\Pb$ and $\Qb$ is informationally possible if and only if $s>\min\{ \tfrac{1}{\sqrt{\lambda}}, \sqrt{\alpha} \}$. However, the best known efficient in \cite{MWXY21+} algorithms only achieves strong detection when $s>\sqrt{\alpha}$, suggesting the emergence of information-computation gap when $\lambda>\alpha^{-1}$. This is further supported by the result in \cite{DDL23+}, where the authors provide evidences that all algorithms based on \emph{low-degree polynomials} fails to strongly distinguish $\Pb$ and $\Qb$ when $s<\sqrt{\alpha}$. 
    \item {\bf Detection-recovery gap.} Another interesting feature is that partial recovery of the latent matching $\pi_*$ (i.e., produce an estimator $\widehat{\pi}$ that agrees with $\pi_*$ on a positive fraction of vertices) is only informationally possible when $s>\tfrac{1}{\sqrt{\lambda}}$, as shown in \cite{DD23b}. This yields that this model has detection-recovery gap when $\lambda<\alpha^{-1}$. Also, we note that when $\lambda<\alpha^{-1}$ the recovery threshold is \emph{a constant factor larger than} the detection threshold, which is unusual in high-dimensional statistics.
    \item {\bf Comparison to previous work \cite{DD23a}.} In a previous work \cite{DD23a}, the authors establish the sharp detection threshold for the detection problem when the average degree $\lambda=n^{1-c+o(1)}$ for some constant $0<c<1$. Our work will follow the basic approach in \cite{DD23a}, where the idea is to calculated the second moment of the likelihood ratio $\tfrac{\mathrm{d}\Pb}{\mathrm{d}\Qb}$ truncated on some good events. However, compared with \cite{DD23a} this work overcomes several technical issues that arise in the constant degree regime. Firstly, instead of truncate on a high-probability event we need to condition on an event that occurs with positive probability, which will enable a more precise enumeration on the number of graphs satisfying certain properties. Secondly, we need to analyze the ``short'' edge orbit more carefully to obtain a sharp strong-detection transition.
\end{itemize}

\subsection{Future directions}

A number of interesting directions remain open and we detail a few here.

\underline{Robust detection.} In this work we determine the sharp strong detection threshold between a pair of correlated \ER graphs and a pair of independent \ER graphs in the constant-degree regime. However, it remains an interesting question whether strong detection is still possible if we adversarially corrupt $o(n)$ edges. In fact, it is straightforward to check that when $s<\tfrac{1}{\sqrt{\lambda}}$, the densest subgraph statistics proposed in \cite{DD23a} is robust up to $o(n)$ adversary edge corruptions. However, when $\tfrac{1}{\sqrt{\lambda}}<s<\sqrt{\alpha}$ it is not known whether we can modify the tree counting statistics in \cite{MWXY21+} to obtain a robust detecting statistics.

\underline{Partially correlated graphs.} Proposed in \cite{HSY24}, another interesting variant for the correlated \ER model is the \ER graph model \emph{with partial vertex correspondence}, where the two graphs are only correlated on a hidden vertex subset. In \cite{HSY24}, the authors determine the detection threshold between this model and a pair of independent \ER graphs up to a constant factor. It remains an interesting question whether the techniques proposed in \cite{DD23a} and in this work might be helpful to sharpen this gap.

\underline{Other graph models.} Another important direction is to understand computational phase transitions for matching other important correlated random graph models, such as the random geometric graph model \cite{WWXY22, GL24+}, the random growing graph model \cite{RS22} and the stochastic block model \cite{RS21, CDGL24+}. We emphasize that it is also important to propose and study correlated graph models based on important real-world and scientific problems, albeit the models do not appear to be ``canonical'' from a mathematical point of view.

\subsection{Notations}

In this subsection, we record a list of notations that we shall use throughout the paper. Recall that $\Pb,\Qb$ are two probability measures on pairs of random graphs on $[n]=\{ 1,\ldots,n \}$. Denote $\Pb_*$ to be the joint law of $(\pi_*,G,A,B)$. Recall that we define $\operatorname{TV}(\mu,\nu)$ to be the total variational distance between two probability distributions $\mu$ and $\nu$, namely
\begin{equation}{\label{eq-def-TV-distance}}
    \operatorname{TV}(\mu,\nu) := \inf_{\gamma \in \Gamma(\mu,\nu)} \mathbb{P}_{(x,y)\sim\gamma} \big[ x \neq y \big] \,, 
\end{equation}
where $\Gamma(\mu,\nu)$ is the set of all couplings between $\mu$ and $\nu$. 

Throughout our discussions, we use $\mathcal K_n$ to denote the complete graph on vertex set $[n]$ and the notation $S \subset \mathcal K_n$ denotes a graph within $[n]$ which has \emph{no isolated vertex}. The vertex set and edge set of $S$ are represented as $V(S)$ and $E(S)$, respectively. Also, denote $\mathfrak S_n$ to be the set of all permutations on $[n]$. For a permutation $\pi\in \mathfrak S_n$, $\pi(S)$ represents the graph with vertex set $\{ \pi(v):v \in V(S) \}$ and edge set $\{(\pi(u),\pi(v)):(u,v) \in E(S)\}$. For $S,T \subset \mathcal K_n$, we say $S \subset T$ if $E(S) \subset E(T)$. Moreover, let $S \cup T \subset \mathcal K_n$ denote the graph induced by the edges in $E(S) \cup E(T)$. Additionally, we define $\mathcal H$ as the set of isomorphism classes of simple graphs with no isolated vertices. For each isomorphic class in $\mathcal H$, we identify it with a specific graph $\mathbf H$ in this class (we use the boldface font to indicate our emphasis on the isomorphism class). We denote by $\mathsf{Aut}(\mathbf H)$ the number of automorphisms of $\mathbf H$ and by $\mathsf{Sub}_n(\mathbf H)= \frac{n!}{(n-|V(\mathbf H)|)! \mathsf{Aut}(\mathbf H)}$ to be the number of inclusions of $\mathbf H$ into $\mathcal K_n$. In addition, for a set $\mathtt A$, we denote by both $|\mathtt A|$ and $\#\mathtt A$ its cardinality. The indicator function of an event $E$ is denoted as $\mathbf 1_{E}$.

For any two positive sequences $\{a_n\}$ and $\{b_n\}$, we write equivalently $a_n=O(b_n)$, $b_n=\Omega(a_n)$, $a_n\lesssim b_n$ and $b_n\gtrsim a_n$ if there exists a positive absolute constant $c$ such that $a_n/b_n\leq c$ holds for all $n$. We write $a_n=o(b_n)$, $b_n=\omega(a_n)$, $a_n\ll b_n$, and $b_n\gg a_n$ if $a_n/b_n\to 0$ as $n\to\infty$. We write $a_n=\Theta(b_n)$, if both $a_n=O(b_n)$ and $a_n=\Omega(b_n)$.

\section{Proof of the main theorem}

In this section we present the proof of Result~(1) of Theorem~\ref{MAIN-THM}. Throughout the remaining part of this paper, we may assume that
\begin{equation}{\label{eq-assumption-s}}
    \lambda^2 s < 1-\epsilon \mbox{ and } s <\sqrt \alpha-\epsilon  
\end{equation}
for a small constant $\epsilon \in (0,0.1)$. As discussed in the introduction, our proof relies on the calculation of the conditional second moment on a positive probability event. We first describe the event we will conditioned on. To this end, let 
\begin{equation}{\label{eq-def-intersection-graph}}
    \mathcal H_{\pi_*} := \Big\{ (i,j) \in \operatorname{U}_n: A_{i,j}= B_{\pi_*(i),\pi_*(j)}=1 \Big\} 
\end{equation}
be the intersection graph. Note that under $\Pb_*$, the distribution of $\mathcal H_{\pi_*}$ is just an \ER graph $\mathcal G(n,\tfrac{\lambda s^2}{n})$. 

\begin{lemma}{\label{lem-no-cycles}}
Let $\mathcal A$ be the event that there are no cycles in $H_{\pi^*}$, then there exist a constant $c=c(\epsilon,\lambda,s)>0$ that for any $n\geq 3$, we have $\Pb_*(\mathcal A)>c$.
\end{lemma}
\begin{proof}
    Note that the number of $k$-cycles in $\mathcal K_n$ equals $\mathtt M_k=\mathtt M_k(n)=\frac{1}{2k} \prod_{i=1}^{n}(n-i+1)$. List all the $k$-cycles in $\mathcal K_n$ in an arbitrary but prefixed order $C^{(k)}_1,\ldots,C^{(k)}_{\mathtt M_k}$, and let
    \begin{equation*}
        \mathcal A^{(k)}_i := \Big\{ C^{(k)}_i \not \in \mathcal H_{\pi_*} \Big\} \mbox{ for } 1 \leq i \leq \mathtt M_k \,.
    \end{equation*}
    One can easily check that $\Pb_*(\mathcal A^{(k)}_{i})=1-(\tfrac{\lambda s^2}{n})^k$ and $\mathcal A_{i}^{(k)}$ are decreasing events. Thus, using FKG inequality we have
    \begin{align*}
        \Pb_*(\mathcal A) = \Pb_*\Big( \cap_{k=3}^{n} \cap_{i=1}^{\mathtt M_k} \mathcal A_{i}^{(k)} \Big)  \geq \prod_{k=3}^{n} \prod_{i=1}^{\mathtt M_k} \Pb_*\big( \mathcal A_{i}^{(k)} \big) = \prod_{k=3}^{n} \Big( 1-(\tfrac{\lambda s^2}{n})^k \big)^{\mathtt M_k}  \\
        \geq \exp\Big( -\sum_{k=3}^{n} \mathtt M_k \cdot (\tfrac{\lambda s^2}{n})^k \big) \geq \exp\Big( -2\sum_{k=3}^{n}(\lambda s^2)^k \Big)
        = e^{-\frac{2(\lambda s^2)^3}{1-\lambda s^2}} >0 \,.
    \end{align*} 
    Recall \eqref{eq-assumption-s}, by taking $c=e^{-\frac{2(\lambda s^2)^3}{\epsilon}}$ yields the desired result.
\end{proof}

Based on Lemma~\ref{lem-no-cycles}, the main goal in our proof is to show the following proposition.
\begin{proposition}{\label{prop-truncate-second-moment}}
    We have $\mathbb E_{\Pb}\Big[ \frac{\Pb(A,B)}{\Qb(A,B)} \cdot \mathbf 1_{\mathcal A} \Big]=O(1)$.
\end{proposition}
We can now finish the proof of Result~(1), Theorem~\ref{MAIN-THM} based on Proposition~\ref{prop-truncate-second-moment}.
\begin{proof}[Proof of Result~(1), Theorem~\ref{MAIN-THM} assuming Proposition~\ref{prop-truncate-second-moment}]
    Note that
    \begin{align*}
        1-\operatorname{TV} (\Pb,\Qb) &= \mathbb E_{(A,B)\sim \Pb} \Big[ \min\Big\{ \tfrac{\Qb(A,B)}{\Pb(A,B)},1 \Big\} \Big] \geq \mathbb E_{(A,B)\sim\Pb} \Big[ \mathbf{1}_{\mathcal A} \cdot \min\Big\{ \tfrac{\Qb(A,B)}{\Pb(A,B)},1 \Big\} \Big]  \\
        &= \mathbb E_{(A,B)\sim \Pb} \Big[ \mathbf{1}_{\mathcal A \cap \{ \Pb(A,B)<\Qb(A,B) \} } + \mathbf{1}_{\mathcal A \cap \{ \Pb(A,B)\geq \Qb(A,B) \} } \cdot \tfrac{\Qb(A,B)}{\Pb(A,B)} \Big]  \,.
    \end{align*}
    Using Cauchy-Schwartz inequality, we get that
    \begin{align*}
        \mathbb E_{(A,B)\sim \Pb} \Big[ \mathbf{1}_{\mathcal A \cap \{ \Pb(A,B)\geq \Qb(A,B) \} } \cdot \tfrac{\Qb(A,B)}{\Pb(A,B)} \Big] &\geq \frac{ \Pb\big( \mathcal A \cap \{ \Pb(A,B) \geq \Qb(A,B) \} \big)^2 }{ \mathbb{E}_{(A,B)\sim \Pb} \big[ \frac{\Pb(A,B)}{\Qb(A,B)} \mathbf{1}_{\mathcal A} \big] } \\
        &\geq \Omega(1) \cdot \Pb\big( \mathcal A \cap \{ \Pb(A,B) \geq \Qb(A,B) \} \big)^2 \,,
    \end{align*}
    where the second inequality follows from Proposition~\ref{prop-truncate-second-moment}.
    Thus, we have 
    \begin{align*}
        1-\operatorname{TV}(\Pb,\Qb) &\geq \Omega(1) \cdot \Bigg( \Pb\Big( \mathcal A \cap \{ \Pb(A,B)< \Qb(A,B) \} \Big) + \Pb\Big( \mathcal A \cap \{ \Pb(A,B) \geq \Qb(A,B) \} \Big)^2 \Bigg) \\
        &\geq \Omega(1) \cdot \min\Big\{ \frac{\Pb(\mathcal A)}{2},\frac{\Pb(\mathcal A)^2}{4} \Big\} = \Omega(1) \,, 
    \end{align*}
    where the second inequality follows from 
    \begin{align*}
        \Pb\Big( \mathcal A \cap \{ \Pb(A,B)< \Qb(A,B) \} \Big) + \Pb\Big( \mathcal A \cap \{ \Pb(A,B) \geq \Qb(A,B) \} \Big) = \Pb(\mathcal A) \,.
    \end{align*}
    Using Lemma~\ref{lem-no-cycles}, there exist a constant $c>0$ that $\Pb(A)>c$ and Thus we have finished the proof.
\end{proof}

The rest part of this section is devoted to the proof of Proposition~\ref{prop-truncate-second-moment}. To this end, denote 
\begin{equation}{\label{eq-def-p-rho}}
    p=\frac{\lambda}{n} \mbox{ and } \rho = \frac{s(1-p)}{1-ps}
\end{equation}
for notational convenience. It is clear that the likelihood ratio $L(A,B)=\frac{\Pb(A,B)}{\Qb(A,B)}$ satisfies that
\begin{align}
    L(A,B) = \frac{1}{n!} \sum_{\pi\in\mathfrak S_n} \frac{\Pb_{(\cdot \mid \pi )}(A,B)}{\Qb(A,B)} = \frac{1}{n!} \sum_{\pi\in\mathfrak S_n} \prod_{(i,j) \in \operatorname{U}_n} \ell(A_{i,j}, B_{\pi(i), \pi(j)}) \,, \label{eq-likelihood-ratio}
\end{align}
where $\ell : \{0,1\} \times \{0,1\} \to \mathbb{R}$ denotes the likelihood ratio function for a pair of correlated edges, given by
\begin{align}{\label{eq-def-ell}}
    \ell(x,y) =
    \begin{cases}
    \frac{1-2ps+ps^2}{(1-ps)^2} \,, & \mbox{if } x=y=0 \,; \\
    \frac{1-s}{1-ps} \,, & \mbox{if } x=1,y=0 \mbox{ or } x=0,y=1 \,; \\
    \frac{1}{p} \,, & \mbox{if } x=y=1 \,.
    \end{cases} 
\end{align}
For any $\pi \in \mathfrak S_n$, define a permutation on $[n]$ by 
\begin{equation}{\label{eq-def-sigma}}
    \sigma = (\pi^{*})^{-1} \circ \pi \,.
\end{equation}
Then $\pi$ (respectively, $\sigma$) induces a bijection $\Pi$ on $\operatorname{U}_n$ (respectively, a bijection $\Sigma$ on $\operatorname{U}_n$), given by $\Pi(i,j) = (\pi(i), \pi(j))$ (respectively, $\Sigma(i,j)) = (\sigma(i), \sigma(j))$. Given $\sigma\in\mathfrak S_n$, for every $e \in \operatorname{U}_n$, we define the orbit generated by $e$ as (below we write $\Sigma^k(e)=\Sigma(\Sigma^{k-1}(e))$ for $e \in \operatorname{U}_n$)
\begin{equation}{\label{eq-def-O-sigma,e}}
    O_{\sigma,e} = \Big\{ \Sigma^k(e): k \in \mathbb{N} \Big\} \,.
\end{equation}
With a slight abuse of notations, we will also regard $O_{\sigma,e}$ as a graph in $\mathcal K_n$ induces by all the edges in $O_{\sigma,e}$.
Note that $\Sigma$ induces an equivalent relation over $\operatorname{U}_n$, i.e., we say $e,e' \in \operatorname{U}_n$ are equivalent if they belong to the same orbit. This equivalent relation decompose $\operatorname{U}_n$ into several equivalent classes, which we denote as $\mathcal E_{\sigma}$, where each $[e] \in \mathcal E_{\sigma}$ represents the equivalent class of $e$. Let  
\begin{equation}{\label{eq-def-O-sigma}}
    \mathcal O_{\sigma}= \Big\{ O_{\sigma,e}: [e] \in \mathcal E_{\sigma} \Big\}
\end{equation}
be the set of edge orbits induced by $\sigma$ (note that $\mathcal O_{\sigma}$ is completely determined by $\sigma$). In addition, define
\begin{equation}
    \mathcal J_\sigma = \{ O \in \mathcal O_\sigma : A_{i,j} = B_{\pi_*(i), \pi_*(j)} = 1 \mbox{ for all } (i,j) \in O \}
\end{equation}
to be the set of edge orbits that are entirely contained in $\mathcal H_{\pi^*}$. It is clear that once $\pi_*$ and $\sigma$ are fixed, the collections of $\{ \ell(A_e,B_{\Pi(e)}): e\in O \}$ are mutually independent for $O\in \mathcal O_{\sigma}$. In addition, we say a graph $J \in \mathcal K_n$ is a \emph{legal realization} of $\mathcal J_{\sigma}$ (denoted as $J {\ltimes} \sigma$), if there exists a subset $\widetilde{\mathcal E}_{\sigma} \subset \mathcal E_{\sigma}$ such that
\begin{align*}
    J = \bigcup_{ [e] \in \widetilde{\mathcal E}_{\sigma} } O_{\sigma,e} \,.
\end{align*}
We first prove the following lemma (similar to \cite[Proposition~3]{WXY23} and \cite[lemma~3.1]{DD23a}) which controls contribution from $O \not \in \mathcal J_{\sigma}$. However, we need to treat the ``short'' orbits (i.e., when $|O|$ is small) more carefully to get a sharper control over the whole region $s\in (0,1)$ (rather than $s$ is a sufficiently small constant as in \cite[Proposition~3]{WXY23} and \cite[lemma~3.1]{DD23a}).

\begin{lemma}{\label{lem-orbit-expectation}}
Recall \eqref{eq-def-p-rho}. For any $\sigma\in\mathfrak S_n$, let $\pi = \pi^* \circ \sigma$ and let $\Pi,\Sigma$ be the bijection on $\operatorname{U}_n$ induced by $\pi,\sigma$, respectively. Then for any legal realization $J$ of $\mathcal J_{\sigma}$ and any orbit $O \in \mathcal O_{\sigma}, O \not \subset J$, we have
\begin{equation}{\label{eq-bound-incomplete-orbit}}
    \mathbb{E}_{(A,B) \sim \Pb(\cdot \mid \pi_*)} \Big[ \prod_{e \in O} \ell(A_e, B_{\Pi(e)}) \mid \mathcal J_{\sigma}=J \Big] = \frac{ 1+\rho^{2|O|}-s^{2|O|} }{ 1-(ps^2)^{|O|} } \,.
\end{equation}
In particular, we have
\begin{equation}{\label{eq-bound-incomplete-orbit-plus}}
    \mathbb{E}_{(A,B) \sim \Pb(\cdot \mid \pi_*)} \Big[ \prod_{e \in O} \ell(A_e, B_{\Pi(e)}) \mid \mathcal J_{\sigma}=J \Big] \leq 
    \begin{cases}
        1+O(n^{-2}) \,, & |O| \geq 2 \,; \\
        1+\frac{\lambda s^2(2s-1)}{n} + O(n^{-2}) \,, & |O|=1 \,.
    \end{cases}
\end{equation}
\end{lemma}

\begin{proof}
Note that $\{ \ell(A_e,B_{\Pi(e)}): e \in O \}$ are independent for $O \in \mathcal O_{\sigma}$. Thus, we have for $O \not \subset J$
\begin{align*}
    \mathbb{E}_{(A,B) \sim \Pb(\cdot \mid \pi_*)} \Big[ \prod_{e \in O} \ell(A_e, B_{\Pi(e)}) \mid \mathcal J_{\sigma}=J \Big] = \mathbb{E}_{(A,B) \sim \Pb(\cdot \mid \pi_*)} \Big[ \prod_{e \in O} \ell(A_e, B_{\Pi(e)}) \mid O \not \subset \mathcal J_{\sigma} \Big]
\end{align*}
We first compute the expectation without conditioning using
\begin{align*}
    \mathbb{E}_{(A,B) \sim \Pb(\cdot \mid \pi_*)} \Big[ \prod_{e \in O} \ell(A_e, B_{\Pi(e)}) \Big] &= \mathbb E_{(A,B) \sim \Qb} \Big[ \frac{ \Pb(A,B \mid \pi_*) }{ \Qb(A,B) } \cdot \prod_{e \in O} \ell(A_e, B_{\Pi(e)}) \Big] \\
    &= \mathbb{E}_{(A,B) \sim \Qb} \big[ \prod_{e \in O} \ell(A_e, B_{\Pi(e)}) \ell(A_{\Sigma(e)}, B_{\Pi(e)}) \Big] \,.
\end{align*}
The right-hand side above can be further interpreted as \( \mathrm{Tr}(\mathcal L^{2|O|}) \), where \( \mathcal L \) is the integral operator on the space of real functions on \( \{0, 1\} \) induced by the kernel \( \ell(\cdot,\cdot) \) in \eqref{eq-def-ell} as
\begin{align*}
    (\mathcal Lf)(x) = \mathbb{E}_{B_e \sim \Qb}[\ell(x, B_e) f(B_e)]
\end{align*}
The matrix form of \( \mathcal L \) is given by \( M(x,y) = \ell(x, y) \cdot \Qb[A_e = y] \) for \( (x, y) \in \{0, 1\} \times \{0, 1\} \), which can be written explicitly as
\[
M =
\begin{pmatrix}
\frac{1 - 2ps + ps^2}{(1 - ps)^2} & \frac{1 - s}{1 - ps} \\
\frac{s}{1 - ps} & \frac{p}{1 - ps}
\end{pmatrix}.
\]
Straightforward calculation yields that \( M \) has two eigenvalues \( 1 \) and \( \rho \). Thus, the expectation without indicator function equals \( 1 + \rho^{2|O|} \) (see \cite[Proposition~1]{WXY23} for details). In addition, it is clear that $\Pb_{*}(O \not \subset \mathcal J_{\sigma})=1-(ps^2)^{|O|}$. Since the condition only excludes the case where \( A_e = B_{\Pi(e)} = 1 \) for all \( e \in O \), we then get that
\begin{equation*}
    \mathbb{E}_{ (A,B) \sim \Pb(\cdot \mid \pi^*) } \Big[ \prod_{e \in O} \ell(A_e,B_{\Pi(e)}) \mid \mathcal J_{\sigma}=J  \Big] = \frac{ 1 + \rho^{2|O|} - s^{2|O|} }{ 1-(ps^2)^{|O|} }  \,, 
\end{equation*}
leading to \eqref{eq-bound-incomplete-orbit}. Based on \eqref{eq-bound-incomplete-orbit}, we have when $|O| \geq 2$ (recalling \eqref{eq-def-p-rho})
\begin{align*}
    \mathbb{E}_{ (A,B) \sim \Pb(\cdot \mid \pi^*) } \Big[ \prod_{e \in O} \ell(A_e,B_{\Pi(e)}) \mid \mathcal J_{\sigma}=J  \Big] \leq \frac{1}{1-O(n^{-2})} = 1+O(n^{-2}) \,,
\end{align*}
and when $|O|=1$
\begin{align*}
    & \mathbb{E}_{ (A,B) \sim \Pb(\cdot \mid \pi^*) } \Big[ \prod_{e \in O} \ell(A_e,B_{\Pi(e)}) \mid \mathcal J_{\sigma}=J  \Big] = \frac{1+\rho^2-s^2}{1-ps^2} \\
    =\ & \Big( 1-s^2 + s^2 (1-\tfrac{\lambda}{n})^2(1+\tfrac{\lambda s}{n} + O(n^{-2}))^2 \Big) \Big( 1+\tfrac{\lambda s^2}{n} + O(n^{-2}) \Big) = 1+\tfrac{\lambda s^2(2s-1)}{n} + O(n^{-2}) \,,
\end{align*}
leading to \eqref{eq-bound-incomplete-orbit-plus}.
\end{proof}

We now give the bound on the probability that $\Pb_*(\mathcal J_{\sigma}=J)$.
\begin{lemma}{\label{lem-prob-realization}}
    Denote $\mathcal O_1(J)=\{ O \in \mathcal O_{\sigma}: |O|=1, O \not\subset J \}$. We have 
    \begin{equation}
        \Pb_{*}(\mathcal J_{\sigma}=J) \leq \big( \tfrac{\lambda s^2}{n} \big)^{|E(J)|} \cdot \big( 1-\tfrac{\lambda s^2}{n} \big)^{|\mathcal O_1(J)|} \,.
    \end{equation}
\end{lemma}
\begin{proof}
    Denote $\mathcal O_{\mathsf{full}}(J) = \{ O \in \mathcal O_{\sigma}: O \subset J \}$. Note that $J = \cup_{O \in \mathcal O_{\mathsf{full}}} O$. Recall that $\{ A_{e}, B_{\Pi(e)}: e \in O \}$ are mutually independent for $O \in \mathcal O_{\sigma}$. We have
    \begin{align*}
        \mathbb P_{*}( \mathcal J_{\sigma} = J ) &\leq \prod_{O \in \mathcal O_{\mathsf{full}}(J)} \mathbb P_{*}(O \subset \mathcal J_{\sigma}) \prod_{O \in \mathcal O_1(J)} \mathbb P_{*}(O \not\subset \mathcal J_{\sigma}) \\
        &= \prod_{O \in \mathcal O_{\mathsf{full}}(J)} \big( \tfrac{\lambda s^2}{n} \big)^{|O|} \cdot \prod_{O \in \mathcal O_1(J)} \big( 1- \tfrac{\lambda s^2}{n} \big) = \big( \tfrac{\lambda s^2}{n} \big)^{|E(J)|} \cdot \big( 1-\tfrac{\lambda s^2}{n} \big)^{|\mathcal O_1(J)|} \,,
    \end{align*}
    where the last equality follows from
    \begin{equation*}
        \sum_{ O \in \mathcal O_{\mathsf{full}}(J) } |O| = |E(J)| \,. \qedhere
    \end{equation*}
\end{proof}

We now move to the proof of Proposition~\ref{prop-truncate-second-moment}. Recall \eqref{eq-likelihood-ratio}. Note that 
\begin{align}
    \mathbb{E}_{(A,B)\sim \Pb} \Bigg[ \mathbf{1}_{\mathcal{A}} \cdot \frac{\Pb(A,B)}{\Qb(A,B)} \Bigg] &= \frac{1}{n!} \sum_{\pi \in \mathfrak S_n} \mathbb{E}_{(A,B)\sim \Pb(\cdot \mid \pi_*)} \Bigg[ \mathbf{1}_{\mathcal{A}} \cdot \prod_{e \in \operatorname{U}_n} \ell(A_e,B_{\Pi(e)}) \Bigg] \nonumber \\
    &=\frac{1}{n!} \sum_{\sigma \in \mathfrak S_n} \mathbb{E}_{(A,B)\sim \Pb(\cdot \mid \pi_*)} \Bigg[ \mathbf{1}_{\mathcal{A}} \cdot \prod_{O \in \mathcal O_{\sigma}} \prod_{e \in O} \ell(A_e,B_{\Pi(e)}) \Bigg]  \,, \label{eq-second-moment-relax-1}
\end{align}
where in the second inequality we use the fact that there is a one-to-one correspondence between $\pi\in\mathfrak S_n$ and $\sigma(\pi) = \pi^{-1} \circ \pi_*$. Recall that $\mathcal{A}$ implies there are no cycles in intersection graph $\mathcal H_{\pi^*}$ and $J_{\sigma}\in \mathcal H_{\pi^*}$. Therefore under $\mathcal A$ the possible realization of $\mathcal J_{\sigma}$ is a forest in $\mathcal K_n$. We then calculate \eqref{eq-second-moment-relax-1} based on the realization $\mathcal J_{\sigma}=J$. Note that $J$ must be a legal realization with respect to $\sigma$. Thus, we have
\begin{align}
    \eqref{eq-second-moment-relax-1} &= \frac{1}{n!} \sum_{\sigma \in \mathfrak S_n} \sum_{ \substack{ J \subset \mathcal K_n: J \ltimes \sigma \\ J \text{ is a forest}} } \mathbb{E}_{(A,B)\sim \Pb(\cdot \mid \pi_*)} \Bigg[ \mathbf{1}_{\mathcal{A}\cap \{\mathcal J_{\sigma}=J\}} \cdot \prod_{O \in \mathcal O_{\sigma}} \prod_{e \in O} \ell(A_e,B_{\Pi(e)}) \Bigg] \nonumber \\
    &\leq \frac{1}{n!} \sum_{\sigma \in \mathfrak S_n} \sum_{ \substack{ J \subset \mathcal K_n: J \ltimes \sigma \\ J \text{ is a forest} } } \mathbb{E}_{(A,B)\sim \Pb(\cdot \mid \pi_*)} \Bigg[ \mathbf{1}_{ \{\mathcal J_{\sigma}=J\} } \cdot \prod_{O \in \mathcal O_{\sigma}} \prod_{e \in O} \ell(A_e,B_{\Pi(e)}) \Bigg] \,. \label{eq-second-moment-relax-2}
\end{align}
Clear we have
\begin{align}
    & \mathbb{E}_{(A,B)\sim \Pb(\cdot \mid \pi_*)} \Bigg[ \mathbf{1}_{ \{\mathcal J_{\sigma}=J\} } \cdot \prod_{O \in \mathcal O_{\sigma}} \prod_{e \in O} \ell(A_e,B_{\Pi(e)}) \Bigg] \nonumber \\
    =\ & \Pb\big( \mathcal J_{\sigma}=J \mid \pi_* \big) \cdot \mathbb{E}_{(A,B)\sim \Pb(\cdot \mid \pi_*)} \Bigg[ \prod_{O \in \mathcal O_{\sigma}} \prod_{e \in O} \ell(A_e,B_{\Pi(e)}) \mid \mathcal J_{\sigma}=J \Bigg] \label{eq-second-moment-relax-3} \,.
\end{align}
Using Lemmas~\ref{lem-orbit-expectation}, we have (note that $|\mathcal O_{\sigma}| \leq n^2$)
\begin{align}
    & \mathbb{E}_{(A,B)\sim \Pb(\cdot \mid \pi_*)} \Bigg[ \prod_{O \in \mathcal O_{\sigma}} \prod_{e \in O} \ell(A_e,B_{\Pi(e)}) \mid \mathcal J_{\sigma}=J \Bigg] \nonumber \\
    =\ & \prod_{O \in \mathcal O_{\mathsf{full}} } \big( \tfrac{n}{\lambda} \big)^{|O|} \cdot \prod_{O \in \mathcal O_{1}(J) } \Big(1+ \tfrac{\lambda s^2(2s-1)}{n} + O(n^{-2}) \Big) \prod_{ O \in \mathcal O_{\sigma} \setminus (\mathcal O_{\mathsf{full}} \cup \mathcal O_{1}(J)) } (1+O(n^{-2})) \nonumber \\
    =\ & O(1) \cdot \big( \tfrac{n}{\lambda} \big)^{|E(J)|} \cdot \Big( 1+\tfrac{\lambda s^2(2s-1)}{n} \Big)^{|\mathcal O_1 (J)|} \,. \label{eq-second-moment-relax-4}
\end{align}
Plugging \eqref{eq-second-moment-relax-4} and Lemma~\ref{lem-prob-realization} into \eqref{eq-second-moment-relax-3}, we have
\begin{align}
    \eqref{eq-second-moment-relax-3} &\leq O(1) \cdot \big( \tfrac{\lambda s^2}{n} \big)^{|E(J)|} \big( 1-\tfrac{\lambda s^2}{n} \big)^{|\mathcal O_1 (J)|} \cdot \big( \tfrac{n}{\lambda} \big)^{|E(J)|} \Big(1+ \tfrac{\lambda s^2(2s-1)}{n} \Big)^{|\mathcal O_1 (J)|} \nonumber \\
    &\leq O(1) \cdot s^{2|E(J)|} \,. \label{eq-second-moment-relax-5}
\end{align}
Plugging \eqref{eq-second-moment-relax-5} into \eqref{eq-second-moment-relax-2}, note that $J \ltimes \sigma$ is equivalent to $\Sigma(J)=J$, we get that
\begin{align}
    \eqref{eq-second-moment-relax-2} &\leq O(1) \cdot \frac{1}{n!} \sum_{\sigma \in \mathfrak S_n} \sum_{ \substack{ J \subset \mathcal K_n: J \ltimes \sigma \\ J \text{ is a forest}  } } s^{2|E(J)|} = O(1) \cdot \frac{1}{n!} \sum_{ \substack{ J \subset \mathcal K_n\\ J \text{ is a forest} } } \sum_{ \sigma\in \mathfrak S_n: \Sigma(J)=J } s^{2|E(J)|} \nonumber \\
    &= O(1) \cdot \frac{1}{n!} \sum_{ \substack{ J \subset \mathcal K_n \\ J \text{ is a forest} } } s^{2|E(J)|} \cdot \#\big\{ \sigma\in\mathfrak S_n: \sigma(J)=J \big\} \,. \label{eq-second-moment-relax-6}
\end{align}
Note that 
\begin{align*}
    \#\big\{ \sigma\in\mathfrak S_n: \sigma(J)=J \big\} = (n-|V(J)|)! \cdot \operatorname{Aut}(J) = \frac{n!}{\mathsf{sub}_n(J)} \,.
\end{align*}
Thus, we have
\begin{align*}
    \eqref{eq-second-moment-relax-6} &= O(1) \cdot \sum_{ \substack{ J \subset \mathcal K_n\\ J \text{ is a forest} } } \frac{s^{2|E(J)|}}{\mathsf{sub}_n(J)} = O(1) \cdot \sum_{ \substack{ \mathbf J \in \mathcal U: |E(\mathbf J)| \leq n \\ \mathbf J \text{ is a forest}  } } \sum_{J \subset \mathcal K_n: J \cong \mathbf J} \frac{s^{2|E(J)|}}{\mathsf{sub}_n(J)} \\
    &= O(1) \cdot \sum_{ \substack{ \mathbf J \in \mathcal U:|E(\mathbf J)| \leq n \\ \mathbf J \text{ is a forest}  } } \frac{s^{2|E(\mathbf J)|}}{\mathsf{sub}_n(\mathbf J)} \cdot \#\big\{ J \subset \mathcal K_n: J \cong \mathbf J \big\} = O(1) \cdot \sum_{ \substack{ \mathbf J \in \mathcal U: |E(\mathbf J)| \leq n \\ \mathbf J \text{ is a forest}  } } s^{2|E(\mathbf J)|} \,.
\end{align*}
Now, to show Proposition~\ref{prop-truncate-second-moment}, it suffices to show the following Lemma.

\begin{lemma}
    For $t<\alpha-\epsilon$, we have
    \begin{align*}
        \sum_{  \mathbf J \in \mathcal U: \mathbf J \text{ is a forest} } t^{|E(\mathbf J)|} = O(1) \,.
    \end{align*}
\end{lemma}
\begin{proof}
    Note that any forest can be decomposed into the union of disjoint trees. Denote $f_k$ to be the number of unlabeled $k$-trees and list all unlabeled trees as $\{ \mathbf T_{k,i} : 1 \leq i \leq k \}$. Suppose $\mathbf J$ can be decomposed into $c_{k,i}$ copies of unlabeled trees $\mathbf T_{k,i}$ with $1 \leq i \leq f_k, k \geq 1$. Then we have
    \begin{align*}
        \sum_{k \geq 0} \sum_{i=1}^{f_k} c_{k,i} (k-1) = |E(\mathbf J)| \,.
    \end{align*}
    Thus, we have that
    \begin{align*}
        \sum_{  \mathbf J \in \mathcal U: \mathbf J \text{ is a forest} } t^{|E(\mathbf J)|} &= \sum_{ k \geq 2 } \sum_{ c_{k,i} \geq 0: 1 \leq i \leq f_k } t^{ \sum_{k \geq 1} \sum_{i=1}^{f_k} (k-1)c_{k,i} } \\
        &= \prod_{k=2}^{\infty} \Big( \frac{1}{1-t^{(k-1)}} \Big)^{f_k} \leq \exp\Big( 2\sum_{k=2}^{\infty} t^{(k-1)} f_k \Big) \,.
    \end{align*}
    Using \cite{Otter48}, we see that 
    \begin{align*}
        f_k = O(1) \cdot \alpha^{-k}/k^{1.5} \mbox{ as } k \to \infty \,.
    \end{align*}
    Thus, for $t<\alpha-\epsilon$ we have $\sum_{k=2}^{\infty} t^{(k-1)} f_k = O(1)$, finishing the proof of this lemma.
\end{proof}

\section*{Acknowledgment}
The author thanks Zhangsong Li for introducing this problem, as well as for insightful discussions and valuable suggestions that greatly improved this manuscript.

\bibliographystyle{plain}

\end{document}